\documentclass{amsproc}
\usepackage{amsmath}
\usepackage{amssymb}

\newtheorem{theorem}{Theorem}[section]
\newtheorem{lemma}[theorem]{Lemma}

\theoremstyle{definition}
\newtheorem{definition}[theorem]{Definition}

\newcommand{\SL}{\text{SL}}
\theoremstyle{remark}
\newtheorem{remark}[theorem]{Remark}

\numberwithin{equation}{section}

\begin{document}

\title{Asymptotic bounds for special values of shifted convolution Dirichlet series}

\author{Olivia Beckwith}
\address{Department of Mathematics and Computer Science \\ Emory University}
\email{olivia.dorothea.beckwith@emory.edu}


\begin{abstract}
In \cite{HoHu}, Hoffstein and Hulse defined the \emph{shifted convolution series} of two cusp forms by ``shifting" the usual Rankin-Selberg convolution $L$-series by a parameter $h$. We use the theory of harmonic Maass forms to study the behavior in $h$-aspect of certain values of these series and prove a polynomial bound as $h \to \infty$. Our method relies on a result of Mertens and Ono \cite{MeOn}, who showed that these values are Fourier coefficients of mixed mock modular forms. 
\end{abstract}
\maketitle
\section{Introduction}
Let $f_1,f_2 \in S_{k} ( \Gamma_0 (N)) $ be cusp forms with $L$-series given by 
$$
L(f_i,s) = \sum_{n=1}^{\infty} \frac{a_i (n) }{n^s}, \quad i = 1,2.
$$
Rankin and Selberg independently defined the \emph{Rankin-Selberg convolution series} $L(f_1 \otimes f_2, s)$ as
$$
L(f_1 \otimes f_2, s) := \sum_{n=1}^{\infty} \frac{a_1 (n) \overline{a_2 (n)}}{n^s} 
$$
for $\Re (s) > k$ and by analytic continuation elsewhere. Rankin-Selberg convolution series were first used to bound Fourier coefficients of cusp forms in the direction of the Ramanujan conjecture, and the idea has also been important in studying the Langlands program. Selberg \cite{Selberg} later defined \emph{shifted convolution $L$-functions}, which have been important in studying the Lindel{\"o}f hypothesis.

In \cite{HoHu} Hoffstein and Hulse defined \emph{shifted convolution series} as follows:
\begin{equation}\label{eq:shiftedseries}
D(f_1, f_2, h;s) := \sum_{n=1}^{\infty} \frac{a_1 (n+h)\overline{a_2(n)}}{n^s}.
\end{equation}
Hoffstein and Hulse established meromorphic continuation for this series and used it to prove strong estimates for certain shifted sums (see Theorem 1.3 of \cite{HoHu}). From these estimates a subconvexity bound for Dirichlet character twists of modular $L$-functions was obtained. 

Shifted convolution sums such as the ones in \cite{HoHu} arise frequently in the theory of automorphic $L$-function and have been studied by many authors, who often use them to prove subconvexity bounds. Duke, Friedlander, and Iwaniec \cite{DuFrIw} were the first to study bounds for shifted convolution sums of Hecke eigenvalues for holomorphic forms and their applications to subconvexity estimates. Harcos \cite{Harcos} extended their work to similar results for Maass forms. Works of Blomer, Harcos, and Michel \cite{Blomer} \cite{BlHaMi} extended the work of \cite{Harcos}, proving a Burgess-type estimate in the latter paper. Note that Blomer \cite{Blomer} showed, as a corollary of his main result, that if  $\epsilon >0$ is fixed, and $h \le M^{\frac{64}{39}- \epsilon}$, there exists $\delta >0$ such that
\begin{equation}\label{eq:Blomer}
\sum_{m \le M} a_1 (m) \overline{a_1(m+h)} \ll_{\epsilon} M^{1 -\delta}.
\end{equation}
\begin{remark}
The convergence of the sum considered in the present work (see equation \ref{eq:defD}) is implied by equation (\ref{eq:Blomer}). Although Blomer only states his result for the case that $f_1 = f_2$, his argument can be extended to the case that $f_1 \neq f_2$.
\end{remark}
These results were extended using automorphic spectral decomposition by Blomer and Harcos in \cite{BlHa} and \cite{BlHa3}. In \cite{BlHa}, a sum very similar to the one studied in \cite{HoHu} and in the present work was considered. In \cite{BlHa3}, a Burgess-type estimate was obtained. Maga \cite{Maga} \cite{Maga2} generalized the bound and Burgess-type bound of \cite{BlHa3} to automorphic $\text{GL}_2$ twisted $L$-functions over general number fields (note that Maga was not the first to obtain a Burgess-type estimate in this generality, but the first to do so using shifted convolution sums). For an overview of these results and their applications to quadratic forms, see \cite{Harcos2}.

We consider \emph{symmetrized shifted convolution series} $\hat{D} (f_1, f_2, h; s)$ for $f_1, f_2 \in S_k(\Gamma_0(N))$, which were first defined by Mertens and Ono \cite{MeOn}. They are defined as follows:
\begin{equation}\label{eq:defD}
\hat{D} (f_1, f_2, h;s) := D(f_1, f_2, h;s) - D( \overline{f_1}, \overline{ f_2}, -h; s).
\end{equation}
This symmetrized series has conditional convergence at $s = k-1$. 

In view of the works described above, it is natural to ask for bounds for the $L$-values in $h$-aspect. Here we use the theory of harmonic Maass forms to obtain a polynomial bound in $h$ for 
$\hat{D} (f_1, f_2, h, k-1)$ as $h \to \infty$. 
\begin{theorem}\label{thm:maintheorem}
Let $f_1,f_2 \in   S_k (\SL_2 (\mathbb{Z}))$.  Then 
$$ |\hat{D} (f_1, f_2, h, k -1)| \ll_{f_1, f_2} h^{\frac{k}{2}}, \quad h \to \infty.$$
\end{theorem}
\begin{remark}
Our methods would also work for forms of higher level, but for simplicity we only do the level 1 case here. Additionally, by making use of the full strength of theorem of Mertens and Ono which involves the Rankin-Cohen bracket, our methods could probably be generalized to the case that the weight of $f_1$ is greater than the weight of $f_2$, rather than their weights being equal.
\end{remark}

In Section 2, we briefly go over the necessary ingredients of our theorem: we recall some important properties of harmonic Maass forms in Section \ref{hmf}, we briefly discuss the Eichler-Shimura isomorphism theorem and period polynomials in Section \ref{EichlerShimura}, and we describe the work of Mertens and Ono \cite{MeOn} in Section 2.3. In Section 3 we prove Theorem \ref{thm:maintheorem} and give an example.
\section*{Acknowledgements}
The author thanks the NSF for its support, and thanks Ken Ono and Michael Mertens for several helpful conversations about harmonic Maass forms. Additionally, the author is grateful to the referee and also Gergely Harcos for their useful comments.

\section{Preliminaries}
\subsection{Harmonic Maass Forms}\label{hmf}
A harmonic Maass form is a certain kind of nonholomorphic modular form that has a natural decomposition into a holomorphic and nonholomorphic part. The holomorphic part of a harmonic Maass form is called a mock modular form, and every mock modular form is naturally associated to a cusp form called its shadow. In this section, we define level 1 harmonic Maass forms (harmonic Maass forms of level greater than 1 are defined as the natural generalization of the definition given here) and state some of their important properties. For more on mock modular forms and harmonic Maass forms, see references such as \cite{BrFu} and \cite{Ono}.

Let $\mathbb{H} := \{ x + i y \in \mathbb{C}: y > 0 \}$ be the upper half plane and let $\SL_2 ( \mathbb{Z})$ be the group of $2 \times 2$ determinant one matrices with integer entries. Every $\gamma =  \left( \begin{array}{cc} a & b \\ c & d \ \end{array} \right) \in \SL_2 (\mathbb{Z})$ has an associated action on $\mathbb{H}$ given by
$$ 
\gamma z = \frac{az + b}{cz+d}.
$$
We define the operator $|_k \gamma$ on smooth functions $f : \mathbb{H} \to \mathbb{C}$ by 
$$
(f|_k \gamma ) (\tau) := (c \tau + d)^{-k} f ( \gamma \tau).
$$

The \emph{weight $k$ hyperbolic Laplacian operator} $\Delta_k$ is defined as follows:
$$
\Delta_k := - y^2 \left( \frac{ \partial^2}{\partial x^2} + \frac{\partial^2}{\partial y^2} \right) + i ky \left( \frac{\partial }{\partial x} + i \frac{ \partial }{ \partial y} \right) . 
$$

Bruinier and Funke in \cite{BrFu} first defined harmonic weak Maass forms. 
\begin{definition}
Let $f: \mathbb{H} \to \mathbb{C}$ be real-analytic, and assume that $k \ge 2$ is an even integer. Then $f$ is a \emph{weight $2-k$ harmonic weak Maass form} if the following hold:

(i) $f$ is weight $2-k$ invariant under $\SL_2 (\mathbb{Z})$, that is, for all $\gamma \in \SL_2(\mathbb{Z})$, $(f |_{2-k} \gamma) = f$. 

(ii) The weight $2-k$ hyperbolic Laplacian operator annihilates $f$, that is, $\Delta_{2-k} f = 0$.  

(iii) There is a polynomial $P_f(\tau) \in \mathbb{C} [ q^{-1} ]$ such that $f(\tau) - P_f( \tau ) = O( e^{- \epsilon y})$ as $y \to \infty$. 

\end{definition}

We let $H_{2-k}$ denote the vector space of weight $2-k$ harmonic weak Maass forms. We also let $M_k := M_k( \SL_2 (\mathbb{Z}) )$ (resp. $S_k := S_k (\SL_2(\mathbb{Z}))$) denote the usual space of modular (resp. cusp) forms of weight $k$ with respect to $\SL_2 (\mathbb{Z})$. For convenience, we refer to harmonic weak Maass forms as \emph{harmonic Maass forms}, and omit the word ``weak".

The following important fact is well known (see \cite{BrFu}). 

\begin{theorem}\label{thm:mockmodular}
Every $f \in H_{2-k}$ can be written in the following way:
$$
f ( \tau) = f^+ ( \tau) + \frac{(4 \pi y)^{1-k}}{k-1} \overline{c_0 (y) } + f^- ( \tau),
$$
where $f^+$ and $f^-$ have Fourier expansions as follows, for some $m_0 \in \mathbb{Z}$:
$$
f^+ ( \tau) = \sum_{n = m_0 }^{\infty} c_f^+ (n) q^n,
$$
and
$$
f^- (\tau) = \sum_{\substack{n > 0 }} \overline{c_f^- (n)}  \Gamma ( 1-k, 4 \pi n y) q^{-n}.
$$
\end{theorem}
In the theorem, $f^+$ is called the \emph{holomorphic part} of $f$, and $\frac{(4 \pi y)^{1-k}}{k-1} \overline{c_0 (y) } + f^- ( \tau)$ is called the \emph{nonholomorphic part} of $f$. When the nonholomorphic part is nontrivial, $f^+$ is called a \emph{mock modular form}.

The following theorem, due to Bruinier and Funke \cite{BrFu}, explains why we conjugate the coefficients of the nonholomorphic part in Theorem \ref{thm:mockmodular}. 
\begin{theorem}[Bruinier, Funke]
The operator $\xi_{2-k} : H_{2-k} \to S_k $  given by $\xi_{2-k}    = 2i y^{2-k} \frac{ \overline{\partial}}{\partial \overline{\tau}}$ is well defined and surjective. Moreover, for $f \in H_{2-k} $,
$$
\xi_{2-k} ( f)  = -( 4 \pi)^{k-1} \sum_{n>0} c_f^- (n) n^{k-1} q^n \in S_k  ,
$$
where $c_f^- (n)$ and $n_0$ are as in Theorem \ref{thm:mockmodular}.
\end{theorem}

For any $F \in H_{2 -k} $, the cusp form $ (- 4 \pi)^{k-1} \sum_{n =n_0 }^{\infty} c_f^-(n) q^n \in S_k $  is called the \emph{shadow} of the mock modular form $F^+$. We say that $F$ is \emph{good} for $f$ if $f$ is the shadow of $F^+$ and $F(\tau) f(\tau)$ is bounded at all cusps. Note that there are many cusp forms $f$ for which there is no mock modular form that is good for $f$.

\subsection{Period Functions}\label{EichlerShimura}
We require some facts about period polynomials and their relationship to the obstructions to modularity for mock modular forms. We first recall the definition and important properties of periods polynomials.

\begin{definition}[Period Polynomial]
Let $f \in S_{k} (\Gamma_0 (N))$ be a cusp form where $k \ge 2$ is even. We define the \emph{$n$th period of $f$} by 
$$
r_n (f) := \int_{0}^{i \infty} f( it ) t^n dt.
$$
 
The period polynomial of $f$ is defined by
$$
r(f;z) := r^+ (f, z) + i r^- (f,z),
$$
where
$$
r^- (f,z) = \sum_{\substack{ 0 \le n \le k -2 \\ 2 \nmid n}} (-1)^{\frac{n-1}{2}} {{k-2}\choose{n}} r_n(f) z^{k-2-n}
$$ 
and
$$
r^+ (f, z) = \sum_{\substack{ 0 \le n \le k-2 \\ 2 | n}} (-1)^{\frac{n}{2}} {{k-2}\choose{n}} r_n (f) z^{k-2-n}.  $$
\end{definition}
\begin{remark}
One can show that $L(f,n+1) = \frac{(2 \pi)^{n+1}}{n!} r_n (f)$, where $L(f,s)$ is the $L$-series for $f$.
\end{remark}

The Eichler-Shimura isomorphism theorem and the work of Kohnen and Zagier \cite{KoZa} imply that the maps $ r^+$ and $r^-$ define correspondences between $S_{k} (\Gamma_0 (N))$ and explicitly defined subspaces of the vector space of degree $k-2$ polynomials with coefficients in $\mathbb{C}$. These important bijections can be used to efficiently compute spaces of cusp forms.

Bringmann, Guerzhoy, Kent, and Ono in \cite{BGKO} connected period polynomials to the theory of harmonic Maass forms. They showed that the obstruction to modularity for a mock modular form can be described in terms of the periods of its shadow. 

\begin{definition}[Period Function]
Let $F^+ (\tau)$ be a mock modular form of weight $2-k$ with respect to $\SL_2 (\mathbb{Z})$, and $\gamma \in \SL_2 ( \mathbb{Z} )$. The \emph{period function of $F^+$ with respect to $\gamma$} is defined as follows:
$$
\mathbb{P} (F^+ , \gamma ; \tau) := \frac{ (4 \pi)^{k-1}}{\Gamma ( k-1)} ( F^+ - F^+|_{2-k} \gamma))(\tau).
$$
\end{definition}

\begin{theorem}[Bringmann, Guerzhoy, Kent, Ono]
Let $S = \left( \begin{array}{cc} 0 & -1 \\ 1 & 0 \end{array} \right)$. Then we have that the period function with respect to $S$ is given by
$$
\overline{\mathbb{P} (F^+, S; \tau )} = \sum_{n=0}^{k-2} \frac{L(f,n+1)}{(k-2-n)!} (2 \pi i \tau)^{k-2-n}. 
$$
\end{theorem}

We require a generalization of this result due to Bringmann, Fricke, and Kent in \cite{BrFrKe}. Among other results, they proved that the period functions corresponding to other modular transformations are also polynomials whose coefficients are essentially values of additive twists of $L(f,s)$. 

Let $L(f, e^{- 2 \pi i d/c};s)$ be defined for $c \neq 0$, $c,d \in \mathbb{Z}$ by 
$$L(f, e^{-2 \pi i d/c}, s):=\sum_{n=1}^{\infty} \frac{e^{-2 \pi i d n /c} a_n }{n^s}$$ for $\Re(s)$ sufficiently large and by analytic continuation elsewhere. The analytic continuation is given by
$$L(f, e^{- 2 \pi i d /c};s) = \frac{(2 \pi)^s}{\Gamma(s)} \int_0^{\infty} f \left(iy - \frac{d}{c} \right) dy.$$ 

\begin{theorem}[Bringmann, Fricke, Kent]\label{thm:CKL}
Let $\gamma = \left( \begin{array}{cc} a & b \\ c & d \end{array} \right) \in \SL_2 (\mathbb{Z})$ satisfy $c \neq 0$. Then
\begin{equation}\label{eq:periodpoly}
  \mathbb{P} ( F^+, \gamma, z)  = \sum_{n=0}^{k-2} \frac{\overline{L(f, e^{-2 \pi i d /c}, n+1 )}}{(k-2-n)!} (-2 \pi i )^{k-2-n} \left(\frac{c z + d}{c}\right)^{k-2-n}. 
\end{equation}
\end{theorem}
\subsection{Work of Mertens and Ono}
Mertens and Ono related the values $\hat{D} (f_1, f_2, h; k-1)$ to the theory of harmonic Maass forms by studying the generating function 
$$\mathbb{L}(f_1, f_2; \tau) := \sum_{h=1}^{\infty} \hat{D} (f_1, f_2, h; k - 1) q^h,
$$
where $q = e^{2 \pi i \tau}$ for $\tau \in \mathbb{H}$. They proved that $\mathbb{L} (f_1,f_2; \tau)$ is the sum of a a weakly holomorphic modular or quasimodular form and the product of a mock modular form and a cusp form. 

To state their theorem, we need to define a few spaces. Let $\widetilde{M}_{k} (\Gamma_0 (N))$ be as follows for even $k \ge 2$:
\begin{equation}\label{eq:holomorphicforms}
\widetilde{M}_k (\Gamma_0(N)) := \left\{ \begin{array}{ll} M_k (\Gamma_0 (N))& \quad \mbox{if $k \ge 4$} \\ M_2(\Gamma_0 (N)) \oplus \mathbb{C} E_2 & \quad \mbox{if $k=2$} \end{array} \right\}
\end{equation}
Moreover, let $\widetilde{M}_k^! (\Gamma_0(N))$ be the extension of $\widetilde{M}_k ( \Gamma_0(N))$ by the weight $k$ weakly holomorphic modular forms on $\Gamma_0 (N)$.  A weakly holomorphic modular form is a meromorphic modular form whose poles are supported at cusps.

\begin{theorem}\label{thm:MertensOno}
For $f_1,f_2 \in S_{k}( \Gamma_0 (N))$,  we have
\begin{equation}\label{eq:Lgenfcn}
\mathbb{L} (f_1 , f_2; \tau) = - \frac{1}{(k - 1)!} M_{f_1}^+ (\tau)  f_2 (\tau)  + F(\tau), 
\end{equation}
where $M_{f_1}^+$ is a mock modular form whose shadow is $f_1$ and $F ( \tau) \in \widetilde{M}_{2}^! (\Gamma_0 (N) )$. If $M_{f_1}$ is \emph{good} for $f_2$, then $F \in \widetilde{M}_{2} (\Gamma_0 (N))$. 
\end{theorem}
\begin{remark}
They actually prove a more general result for when $f_1$ and $f_2$ are cusp forms of weight $k_1$ and $k_2$ respectively with $k_1 \ge k_2$. In this case, the formula involves the Rankin-Cohen bracket $ [M_{f_1}^+ (\tau),  f_2 (\tau)]_{\frac{k_1 - k_2}{2}}$, and the form $F$ lies in $\widetilde{M}^!_{\frac{k_1 - k_2}{2}} ( \Gamma_0 (N))$. 
\end{remark}

The form $F(\tau)$ in Theorem \ref{thm:MertensOno} can be described as the image of $M_{f_1} f_2$ under a modified holomorphic projection operator. Recall that if $f$ is a smooth weight $k \ge 2$ modular form for $\Gamma_0(N)$ with moderate growth at cusps, then its holormorphic projection $\pi_{hol} f$ lies in $\widetilde{M}_k (\Gamma_0(N))$. For more on the classical holomorphic projection operator, see \cite{Sturm}, \cite{ImRaRi}, \cite{Mertens} and \cite{GrZa}.

The regularized holormorphic projection operator $\pi_{hol}^{reg}$ is an extension of $\pi_{hol}$ to an operator on smooth modular forms with certain exponential singularities at cusps. This definition is due to Mertens and Ono \cite{MeOn} who based it on Borcherds' \cite{Borcherds} regularized Petersson inner product. 
\begin{definition}{Regularized Holomorphic Projection}
Let $f: \mathbb{H} \to \mathbb{C}$ be real-analytic, weight $k \ge 2$ modular with respect to $\Gamma_0 (N)$, and have Fourier series $\sum_{n \in \mathbb{Z}} a_f (n,y) q^n$ . Let the cusps of $\Gamma_0(N)$ be denoted $\kappa_1, \cdots, \kappa_s$ where $\kappa_1 = i \infty$. For each $\kappa_j$, fix some $\gamma_j \in \SL_2 (\mathbb{Z})$ with $\gamma_j \kappa_j = i \infty$.  Suppose that for each $\kappa_j$, there is a polynomial $H_{\kappa_j} (X) \in \mathbb{C} [X]$ such that 
$$
(f|_k \gamma_j^{-1})(\tau) - H_{\kappa_j} (q^{-1}) = O(v^{- \epsilon}),
$$
for some $\epsilon >0$. Also, suppose $a_f (n,y) = O(y^{2-k})$ as $y \to 0$ for all $n >0$. Then we define the \emph{regularized holomorphic projection} of $f$ by
$$
( \pi_{hol}^{reg} f) = H_{i \infty} (q^{-1}) + \sum_{n =1}^{\infty} c(n) q^n,
$$
where
$$
c(n) = \lim_{s \to 0} \frac{ (4 \pi n)^{k-1}}{(k-2)!} \int_0^{\infty} a_f(n,y) e^{- 4 \pi n y} y^{k-2-s} dy.
$$
\end{definition} 
It turns out that if $f$ is a real analytic modular form, $\pi_{hol}^{reg} f$ is a weakly holomorphic modular or quasimodular form.
\begin{theorem}[Mertens, Ono]
Suppose $f$ is as in the previous definition. Then $\pi_{hol}^{reg} (f)$ lies in $\widetilde{M}_{k}^! (\Gamma_0(N))$.
\end{theorem}
\begin{remark}
In Theorem \ref{thm:MertensOno}, we have $$F(\tau) = \frac{1}{(k_1 - 1)!} \pi_{hol}^{reg}  (M_{f_1}^+ \cdot f_2 )(\tau).$$
\end{remark}

\section{Proof of Theorem \ref{thm:maintheorem}}
In this section we prove Theorem \ref{thm:maintheorem}. First we prove a lemma which gives a bound for the obstruction to modularity for a mock modular form. Throughout the section, let $\mathcal{F}$ denote the usual fundamental domain for $\mathbb{H}$, given by
$$
\mathcal{F} := \{ z \in \mathbb{H}: |z| >1,  -1 \le \Re(z) <1 \}.
$$ 

\subsection{Lemma}
We prove an estimate for $\mathbb{P} ( M_{f}^+, \alpha; \tau)$ (see Section \ref{EichlerShimura} for the definition). 

\begin{lemma}\label{thm:periodbound}
Let $f \in S_{k}$ be a cusp form, and $M_{f}^+$ be a harmonic Maass form whose shadow is $f$. Then there exists a constant $C(f) >0$ such that for all $\alpha = \left( \begin{array}{cc} a & b \\ c & d \end{array} \right) \in \SL_2 (\mathbb{Z})$ with $c \neq 0$ and $\tau \in \mathcal{F}$, we have
$$|\mathbb{P}(M_f^+, \alpha; \tau)| \le C(f) |c\tau + d|^{k-2}. $$
\end{lemma}
\begin{proof}
By Theorem \ref{thm:CKL}, we have
\begin{equation}\label{eq:lemma1}
\frac{\mathbb{P} ( M_f^+, \alpha; \tau)}{(c \tau + d)^{k-2}}= \sum_{n=0}^{k-2} \frac{\overline{L(f, e^{- 2 \pi i d/c}, n+1)}}{(k-2-n!)} (- 2 \pi i )^{k-2-n} \left( \frac{1}{c^{k-2-n} (c \tau + d)^n} \right).
\end{equation}
One can show that $\int_0^{\infty} f(iy - x) y^{s-1} dy$ is a periodic continuous function in $x$, thus for fixed $n$  the values $L(f, e^{- 2 \pi i d/c};n+2)$ can be bounded  independently of $c$ and $d$. Since $|c| \ge 1$ and $|c \tau + d| \ge \frac{\sqrt{3}}{2}$ for $\tau \in \mathcal{F}$, the right hand side of equation \ref{eq:lemma1} is bounded uniformly in $\alpha$ and $\tau \in \mathcal{F}$. 
\end{proof}
\subsection{Proof of Theorem \ref{thm:maintheorem}}

Let $F( \tau) := \pi_{hol}^{reg} (  M_{f_1}  f_2  ) - M_{f_1}^+   f_2  =  \pi_{hol}^{reg} ( M_{f_1}^-  f_2   )$. By Theorem \ref{thm:MertensOno}, we have $ F(\tau) = (k - 1)! \cdot \mathbb{L}  ( f_1, f_2; \tau)$.

Since $F$ is holomorphic, by Cauchy's integral formula the coefficients of $\mathbb{L} (f_1, f_2, \tau)$ are given by a contour integral as follows.
\begin{align*}
(k-1)! \hat{D}(f_1, f_2, h; k-1) &= \frac{1}{2 \pi i} \int_C \frac{F(\tau) }{q^{h+1}}  d q \\
&= \int_0^1 F\left(x + \frac{i}{h}\right) e^{- 2 \pi i h (x + (i/h))}  d x. 
\end{align*}

Choose $\beta \in \mathbb{C}$ so that $G(\tau) := \pi_{hol}^{reg} (M_{f_1}^- f_2) - \beta E_2 (\tau)$ lies in $M_2^! (\SL_2 (\mathbb{Z}))$, and let $E_2^* (\tau)$ be the completed weight 2 nonholomorphic modular form $E_2^* (\tau) = E_2 (\tau) - \frac{3}{\pi \Im(\tau)}$. We rewrite the integral in the previous expression as follows:
\begin{align*}
(k-1)! \hat{D}(f_1, f_2, h; k-1) &= \int_0^1 e^{- 2 \pi i h (x + (i/h))} \left(\beta E_{2}^* \left(x + \frac{i}{h} \right) + G\left(x + \frac{i}{h} \right)  \right) dx \\
& - \int_0^1 e^{- 2 \pi i h (x + (i/h))} M_{f_1} \left(x + \frac{i}{h} \right) f_2 \left ( x + \frac{i}{h} \right)   d x \\
&+ \int_0^1 M_{f_1}^- \left(x + \frac{i}{h} \right) f_2 \left( x + \frac{i}{h} \right) e^{- 2 \pi i h (x + (i/h))} d x \\
&- \beta \int_0^1 e^{- 2 \pi i h (x + (i/h))} \frac{3}{\Im ( x + \frac{i}{h}) \pi }  d x.
\end{align*}

By direct evaluation, the fourth integral is 0.

The difference of the first and second integrals satisfies an $O(h)$ estimate. This follows from the fact that the difference of the integrands is a smooth weight 2 modular form which vanishes as $e^{2 \pi i \tau}$ as $\tau \to i \infty$. 

To complete the proof, it is sufficient to show that the third integral satisfies an $O(h^{\frac{k}{2}} )$ estimate, and for this it is sufficient to establish that $h(\tau) := M_{f_1}^- ( \tau) f_2 ( \tau)y^{\frac{k}{2}}$ is bounded on $\mathbb{H}$. 

As $\tau \to i\infty$, $h(\tau)$ has exponential decay because of the exponential decay of $f_2$. Thus, $h$ is bounded on the fundamental domain $\mathcal{F}$.

It is sufficient to show that for $\alpha = \left( \begin{array}{cc} a & b \\ c & d \ \end{array} \right) \in \SL_2 (\mathbb{Z})$, $h ( \alpha \tau)$ is bounded on $\mathcal{F}$ uniformly with respect to $\alpha$. 
Rewriting $|h(\alpha \tau)|$ using the modular invariance of $|f_2 (\tau) \Im (\tau)^{\frac{k}{2}} |$, we have

$$
|h(\alpha \tau)| = |f_2 (\tau) M_{f_1}^- (\alpha \tau)| | \Im (\tau ) |^{\frac{k}{2}}.
$$

Substituting 
$$
M_{f_1}^- ( \tau) + \mathbb{P} ( f, \alpha, \tau) =  M_{f_1}^- |_{2-k} (\alpha) (\tau)
$$

gives 
$$
|h (\alpha \tau)| \le  |f_2 (\tau) \Im (\tau) |^{\frac{k}{2}}  \left( |\frac{M_{f_1}^- ( \tau)}{(c \tau + d)^{k-2}}| + |\frac{1}{(c \tau + d)^{k-2}} \mathbb{P} ( f, \alpha; \tau)| \right).
$$

The second factor is bounded on $\mathcal{F}$  because of Lemma \ref{thm:periodbound}, the fact that $|c \tau + d| \ge \frac{\sqrt{3}}{2}$ on $\mathcal{F}$, and the exponential decay of $M_{f_1}^- (\tau)$ as $\tau \to i \infty$.   On the other hand, $f_2 (\tau) |\Im (\tau) |^{\frac{k}{2}}$ has exponential decay at $i \infty$. Thus, $|h ( \alpha \tau)|$ is bounded for $\tau \in \mathcal{F}$. This completes the proof.

\subsection{Example}
When $f_1 = f_2 = \Delta$, where $\Delta( \tau)$ is the modular discriminant, that is, the unique normalized cusp form of weight 12, Theorem \ref{thm:MertensOno} says
$$
\mathbb{L}(\Delta, \Delta; \tau) = \frac{ Q^+ (-1,12,1;\tau) \Delta (\tau)}{11! \beta} - \frac{E_2 (\tau)}{\beta} = 33.38465... q + 266.447... q^2 + \cdots ,
$$
where $Q^+ (-1,12,1;\tau)$ is the holomorphic part of the Maass-Poincare series of weight 12 and level 1 with a simple pole at $i \infty$. It follows from Theorem \ref{thm:maintheorem} that the Fourier coefficients of $\mathbb{L}(\Delta, \Delta; \tau)$ grow as $O(n^6)$. The following table illustrates the significant cancellation that occurs. Here, $c_{\Delta}^+ (n)$ denotes $n$th Fourier coefficient of $Q^+ (-1,12,1;\tau)$, which grows exponentially with $n$.

\begin{table}[h!]
\begin{tabular}{|c|c|c|c|c|}
\hline
$n$ & $1$ & $10$ & $100$  & $1000$ \\
\hline
$c_{\Delta}^+ (n) / 11!$ & $-1842.89....$ & $4.94... 10^{10}$ & $5.19... 10^{42}$ & $1.30... 10^{155}$ 
\\
\hline
$\hat{D}(\Delta,\Delta,n;11)$ & $33.384 ...$ & $538192.6 ...$ &  $80949379532.2...$ & $5.4234... 10^{15}$ \\
\hline
\end{tabular}
\vspace{0.5cm}
\caption{Numerics for Theorem \ref{thm:maintheorem}}
\end{table}


\begin{thebibliography}{GKZ}
\bibitem{BGKO} K. \ Bringmann, P. \ Guerzhoy, Z. \ Kent, K. \ Ono, \emph{Eichler Shimura theory for mock modular forms}, Mathematische Annalen, \textbf{355} (2013),   1085 - 1121. 

\bibitem{Blomer} V. \ Blomer, \emph{Shifted convolution sums and subconvexity bounds for automorphic L-functions}, Int. Math. Res. Notes, \textbf{73} (2004), 3905 - 3926. 

\bibitem{BlHa} V. \ Blomer, G. \ Harcos, \emph{The spectral decomposition of shifted convolution sums}, Duke Math. J. \textbf{144} 
(2008), 321-339.

\bibitem{BlHa2} V. \ Blomer, G. \ Harcos, \emph{Hybrid bounds for twisted $L$-functions}, J. Reine Anfew. Math. \textbf{621} (2008) 53-79.; Addendum, ibid. 694 (2014), 241 - 244. 

\bibitem{BlHa3} V. \ Blomer, G. \ Harcos, \emph{Twisted L-functions over number fields and Hilbert's eleventh problem}, Geom. Funct. Anal. \textbf{20} (2010), 1-52.

\bibitem{BlHaMi} V. \ Blomer, G. \ Harcos, P. \ Michel, \emph{A Burgess-like subconvex bound for twisted L-functions (with Appendix 2 by 
Z. Mao)}, Forum Math. \textbf{19} (2007), 61-105.

\bibitem{Borcherds} R. E. \ Borcherds, \emph{Automorphic forms with singularities on Grassmanians}, Invent. Math. \textbf{132} (1998), 491 - 562.

\bibitem{BrFrKe} K. \ Bringmann, K.-H. \ Fricke, Z. \ Kent, \emph{Special $L$-values and periods of weakly holomorphic modular forms}, Proc. AMS \textbf{142} (2014), 3425 - 3439. 

\bibitem{BrFu} J. H. \ Bruinier, J. \ Funke, \emph{On two geometric theta lifts}, Duke Math J. \textbf{1} (2004), no. 125, 45-90.


\bibitem{CKL} D. \ Choi, B. \ Kim, S. \ Lim, \emph{Eichler integrals and harmonic weak Maass forms}, J. Math. Anal. Appl., \textbf{411}, (2014) 429-441.

\bibitem{DuFrIw} W. \ Duke, D.B. \ Friedlander, H. \ Iwaniec, \emph{Bounds for automorphic $L$-functions}, Invent. Math. \textbf{112}, (1993), 1-8.

\bibitem{GrZa} B.H. \ Gross, D.B. \ Zagier, \emph{Heegner points and derivatives of L-series}, Invent. Math. \textbf{84} (1986), 225-320. 

\bibitem{Harcos} G. \ Harcos, \emph{An additive problem in the Fourier coefficients of cusp forms}, Math. Ann. 
\textbf{326} (2003), 347-365.

\bibitem{Harcos2} G. \ Harcos, \emph{Twisted Hilbert modular L-functions and spectral theory}, Adv. Lect. Math. (ALM), \textbf{30} (2014), 49-67.

\bibitem{HoHu} J. \ Hoffstein, T. A. \ Hulse, \emph{Multiple Dirichlet series and shifted convolutions}

\bibitem{ImRaRi} \"{O}. \ Imamoglu, M. \ Raum, O. \ Richter, \emph{Holomorphic projections and Ramanujan's mock theta functions}, Proc. Nat. Acad. Sci. U.S.A. 111.11 (2014), 3961-3967.

\bibitem{KoZa} W. \ Kohnen, D.B. \ Zagier, \emph{Modular forms with rational periods}, Modular Forms (Durham 1983), 197-249, Ellis Horwood Ser. Math. Appl.: Statist. Oper. Res., Horwood, Chichester, 1984.

\bibitem{LeZa} J. \ Lewis, D.B. \ Zagier, \emph{Period functions for Maass wave forms. I}, Annals of Mathematics \textbf{153} (2001) 191 - 258.

\bibitem{Maga} P. \ Maga, \emph{Shifted convolution sums and Burgess type subconvexity over number fields}, to appear in J. Reine Angew. Math.

\bibitem{Maga2} P. \ Maga, 2013: \emph{Subconvexity and shifted convolutions sums over number fields}. PhD Thesis.  Central European University, 89 pp.  

\bibitem{Mertens} M. \ Mertens, \emph{Eichler-Selberg Type Identities for Mixed Modular Forms}, submitted. arxiv:1404:5491.

\bibitem{MeOn} M. \ Mertens, K. \ Ono, \emph{Special values of shifted convolution Dirichlet series}, Mathematika, \textbf{62} (2016), 47-66.

\bibitem{Ono} K. \ Ono, \emph{Unearthing the visions of a master: harmonic Maass forms and number theory}, Proceedings of the 2008 Harvard-MIT Current Developments in Mathematics Conference, International Press, Somerville, MA, 2009, pages 347-454.

\bibitem{Selberg} A. \ Selberg, \emph{On the estimation of Fourier coefficients of modular forms}, Proc. Sympos. Pure Math. Vol VIII, Amer. Math. Soc. Providence, RI 1965, 1-15.

\bibitem{Sturm} J. \ Sturm, \emph{Projections of $C^{\infty}$ automorphic forms}, Bull. Amer. Math. Soc. (N.S.) \textbf{2} (1980), no. 3, 435-439. 
\end{thebibliography}
\end{document}